\documentclass[11pt,letterpaper]{amsart}
\usepackage[margin=1in]{geometry}

\usepackage{amsmath,amssymb,amsthm,mathrsfs,bbm,comment,units,enumitem,xcolor,adforn,pgfornament,fancyhdr,lastpage}
\usepackage[colorlinks, linktocpage, allcolors=black,breaklinks]{hyperref}

\renewcommand{\leq}{\leqslant}
\renewcommand{\geq}{\geqslant}

\newtheoremstyle{mythm}
{.5\baselineskip}	
{.5\baselineskip}	
{}		
{}		
{\bf}	
{. }		
{ }		
{}		

\theoremstyle{mythm}
\newtheorem{theorem}{Theorem}	
\newtheorem{lemma}[theorem]{Lemma}

\newtheorem{corollary}[theorem]{Corollary}
\newtheorem{definition}[theorem]{Definition}
\newtheorem{example}[theorem]{Example}

\newtheorem{remark}[theorem]{Remark}
\newtheorem{question}{Question}

\newtheoremstyle{myclaim}
{.5\baselineskip}	
{.5\baselineskip}	
{}		
{}		
{\sc}	
{. }		
{ }		
{}		

\theoremstyle{myclaim}

\newcommand{\astfill}{\noindent\xleaders\hbox{$\ast$}\hfill\kern0pt}

\definecolor{SoftRed}{HTML}{DD2222}

\usepackage[british]{babel}

\title{The Hurewicz Property and the Vietoris Hyperspace}

\author[C. Caruvana]{Christopher Caruvana}
\address{School of Sciences,
Indiana University Kokomo,
2300 S. Washington St.,
Kokomo, IN 46902,
United States of America.}
\email{chcaru@iu.edu}
\date{\today}



\begin{document}

\maketitle

\begin{abstract}
    In this note, we characterize when the Vietoris space of compact subsets of a given space
    has the Hurewicz property in terms of a selection principle on the given space itself using \(k\)-covers and the notion
    of groupability introduced by Ko{\v{c}}inac and Scheepers.
    We comment that the same technique establishes another equivalent condition
    to a space being Hurewicz in each of its finite powers.
    We end with some characterizations involving spaces of continuous functions and answer
    a question posed by Ko{\v{c}}inac.
\end{abstract}

\section{Introduction}

The Hurewicz property was introduced in \cite{Hurewicz} and was conjectured by Hurewicz to characterize the
property of being \(\sigma\)-compact.
Though, as discussed in \cite{TsabanTheBook}, this conjecture is false, investigations about
the Hurewicz property continue.
In \cite{COOC7}, Ko{\v{c}}inac and Scheepers introduce the notion of groupability to characterize the Hurewicz property
using other types of selection principles that mirror Menger's property.
An important ingredient in this line of investigation is the game-theoretic characterization of the
Hurewicz property established in \cite{COOC1}.
Also in \cite{COOC7}, a selection principle which is equivalent to every finite power of a space
having the Hurewicz property is established.
This adds to the list of characterizations of certain properties of \(X\) in every finite power;
for other characterizations of this flavor involving traditional selection principles,
see \cite{GerlitsNagy,JustMillerScheepers,SakaiCDoublePrime}.

Characterizations of the Hurewicz property involving hyperspaces have been obtained in recent years.
Characterizations in terms of certain properties on \(C_p(X)\) for certain classes of spaces \(X\) have been done
in \cite{ScheepersSequentialHurewicz} and \cite{OsipovFunctionalHurewicz}.
Following generalizations of \(\pi\)-networks in \cite{Li}, a characterization of the Hurewicz property on certain hyperspaces
of \(X\) in terms of further generalized notions of \(\pi\)-networks of \(X\) can be found in \cite{CruzCastillo}.
Other characterizations can be found in \cite{DasSamantaChandra} and \cite{TsabanBaire}.
For generalized versions of groupability, see \cite{DasSamantaChandra2019}.

In this note, we continue this work by characterizing the Hurewicz property on certain
hyperspaces in terms of selective covering properties on \(X\) using the notion of groupability,
extending on the family of relationships established in \cite{CHVietoris}.
The first result is a general criterion which guarantees the equivalence of the Hurewicz property on a hyperspace with
a certain selective property on the space \(X\).
As a consequence, we characterize the Hurewicz property on the Vietoris space of compact subsets of \(X\)
in terms of \(k\)-covers of \(X\) and groupability.

The original characterization involving finite powers of \(X\) was done using \(\omega\)-covers,
which cover finite subsets of \(X\).
Since finite powers of \(X\) clearly cover finite subsets, relationships between properties of covers of finite powers
and \(\omega\)-covers are not surprising.
However, the finite powers of \(X\) carry more information than finite subsets, alone.
Here, we establish a similar equivalence involving \(\omega\)-covers with the set
of finite subsets of \(X\) adequately topologized.

We end by showing that the selective properties discussed above are also characterized by the property that the space
\(C_p(X)\) or \(C_k(X)\) is Reznichenko and has countable (strong) fan-tightness.

\section{Preliminaries}

Throughout, by \emph{space} we mean \emph{topological space} and follow definitions
and notation of \cite{Engelking}.
For a space \(X\), we will use \([X]^{<\omega}\) to denote the set of all non-empty finite subsets
of \(X\) and \(K(X)\) to denote the set of non-empty compact subsets of \(X\).
To avoid trivialities, we will assume our spaces \(X\) to be infinite or non-compact, wherever appropriate.
We also assume all spaces are Hausdorff.

As usual, for sets \(A\) and \(B\), we will use the notation \(A^B\) to be the set of all functions \(B \to A\).
So, for example, \(X^\omega\) is the set of all sequences in \(X\).
When dealing with a finite power \(\omega^{n+1}\) of \(\omega\) where \(n \in \omega\),
we may use the notation \(m_j\) in place of \(m(j)\) for \(j \in n+1 = \{ 0,1,\ldots, n \}\).
The symbol \(\langle\rangle\) denotes the empty sequence.
We will also use \(X^{<\omega}\) to denote the set \(\bigcup_{n \in \omega} X^{n}\) of finite sequences of \(X\) and,
if the base \(X\) is written with a few symbols, we will use
\((X)^{<\omega}\) in place of \(X^{<\omega}\) to avoid confusion.

By a \emph{topological operator} we mean a function on the class of
topological spaces that produces a collection of sets related to the topological structure of the input.
Some well-known topological operators that we'll be using in this paper are recalled below.
\begin{itemize}
    \item
    \(\mathscr T\): For a space \(X\), \(\mathscr T_X\)
    is the collection of all open sets \(U \subseteq X\) where \(U \neq \emptyset\) and \(U \neq X\).
    \item
    \(\mathcal O\): For a space \(X\), \(\mathcal O_X\) is the
    collection of non-trivial open covers of \(X\); that is, all \(\mathscr U \subseteq \mathscr T_X\)
    with \(X = \bigcup\mathscr U\).
    \item
    \(\Omega\): For a space \(X\), \(\Omega_X\) is the collection of \(\omega\)-covers of \(X\);
    that is, all \(\mathscr U \in \mathcal O_X\) so that,
    \[(\forall F \in [X]^{<\omega})(\exists U \in \mathscr U) \ F \subseteq U.\]
    \item
    \(\mathcal K\): For a space \(X\), \(\mathcal K_X\) is the collection of \(k\)-covers of \(X\);
    that is, all \(\mathscr U \in \mathcal O_X\) so that
    \[(\forall K \in K(X))(\exists U \in \mathscr U) \ K \subseteq U.\]
\end{itemize}
\begin{definition}
    For a space \(X\), let \(\mathbb K(X)\) denote the set \(K(X)\)
    endowed with the Vietoris topology; that is, the topology generated by sets of the form
    \(\{ K \in \mathbb K(X) : K \subseteq U \}\) and \(\{ K \in \mathbb K(X) : K \cap U \neq \emptyset\}\) for \(U \subseteq X\) open.
    For \(U_1 , \ldots , U_n\) open in \(X\), define
    \[[U_1, \ldots, U_n] = \left\{ K \in \mathbb K(X) : K \subseteq \bigcup_{j=1}^n U_j
    \text{ and } (\forall j)\left[ K \cap U_j \neq \emptyset \right] \right\}.\]
    These sets form a basis for the topology on \(\mathbb K(X)\).

    When working with a fixed subspace \(\mathbb Y \subseteq \mathbb K(X)\), we will use
    \([U_1,\ldots, U_n]\) to mean \([U_1,\ldots, U_n] \cap \mathbb Y\).
\end{definition}
For a detailed treatment of the Vietoris topology, see \cite{MichaelTopologies}.
\begin{definition}
    We will let \(\mathcal P_{\mathrm{fin}}(X)\) denote the set \([X]^{<\omega}\) with the topology
    it inherits as a subspace of \(\mathbb K(X)\).
\end{definition}
\begin{definition}
    For a space \(X\), by an \emph{ideal of closed sets}, we mean a collection \(\mathcal A\) of proper closed subsets
    of \(X\) with \([X]^{<\omega} \subseteq \mathcal A\) so that the following two properties hold:
    \begin{itemize}
        \item
        For \(A , B \in \mathcal A\), \(A \cup B \in \mathcal A\) as long as \(A \cup B \neq X\).
        \item
        For \(A \in \mathcal A\), if \(B \subseteq A\) is closed and non-empty, then \(B \in \mathcal A\).
    \end{itemize}
    Such colletions are often referred to as bases for ideals or bornologies (see \cite{BornologyBook}).
\end{definition}
The two ideals of closed sets we'll be most interested in are \([X]^{<\omega}\) and \(K(X)\).
\begin{definition}
    For a space \(X\) and an ideal \(\mathcal A\) of closed sets, we define the set of \(\mathcal A\)-covers,
    denoted by \(\mathcal O_X(\mathcal A)\), to be all \(\mathscr U \in \mathcal O_X\) so that
    \[(\forall A \in \mathcal A)(\exists U \in \mathscr U)\ A \subseteq U.\]
\end{definition}
Note that \(\Omega_X = \mathcal O_X([X]^{<\omega})\) and \(\mathcal K_X = \mathcal O_X(K(X))\).
Moreover, since \([X]^{<\omega}\) and \(K(X)\) can be defined topologically, one could hope for a convenient
operator to be included in our list of operators above.
However, not all ideals of closed sets can necessarily be defined without referencing the particular topology
they are coming from, so we will not be able to drop the subscript for general \(\mathcal O_X(\mathcal A)\).

\begin{remark}
    Observe that, since we define \(\mathcal O_X(\mathcal A)\) to consist of non-trivial open covers
    and \(\mathcal A\) is an ideal of closed sets, each \(\mathscr U \in \mathcal O_X(\mathcal A)\) has the property that
    \(\mathscr U \setminus \mathscr F \in \mathcal O_X(\mathcal A)\) for every \(\mathscr F \in [\mathscr U]^{<\omega}\).
    Moreover, every member of \(\mathcal O_X(\mathcal A)\) is infinite.
\end{remark}
\begin{definition}
    We say that a space \(X\) has the \emph{Hurewicz property} (or that \(X\) is \emph{Hurewicz})
    if, for any sequence \(\langle \mathscr U_n : n \in \omega \rangle \in\mathcal O_X^\omega\),
    there exists a sequence \(\langle \mathscr F_n : n \in \omega \rangle\)
    so that \(\mathscr F_n \in \left[ \mathscr U_n \right]^{<\omega}\) for each \(n \in \omega\) and
    \[(\forall x \in X)(\exists m \in \omega)(\forall n \geq m)\ x \in \bigcup \mathscr F_n.\]
\end{definition}
Note that we are avoiding the use of the commonly used notation \(\mathsf U_{\mathrm{fin}}(\mathcal O, \Gamma)\)
to avoid complications in which \(\bigcup \mathscr F_n = X\) for some \(n \in \omega\) or
the set \(\left\{ \bigcup \mathscr F_n : n \in \omega \right\}\) is finite.

We will also generally be careful to distinguish between sequences and countable sets, as can be observed
in our treatment of the following definition.
\begin{definition}[\cite{COOC1}] \label{def:Selection}
    For topological operators or collections \(\mathcal A\) and \(\mathcal B\),
    \begin{itemize}
        \item
        \(\mathsf{S}_{\mathrm{fin}}(\mathcal A, \mathcal B)\) denotes the following selection principle:
        for each sequence \(\langle \mathscr U_n : n \in \omega \rangle\) of elements of \(\mathcal A\),
        there is a sequence \(\langle \mathscr F_n : n \in \omega \rangle\), where each \(\mathscr F_n\) is a finite subset
        of \(\mathscr U_n\), so that \(\bigcup_{n\in\omega} \mathscr F_n \in \mathcal B\).
        \item
        \(\mathsf S_1(\mathcal A, \mathcal B)\) denotes the following selection principle:
        for each sequence \(\langle \mathscr U_n : n \in \omega \rangle\) of elements of \(\mathcal A\),
        there is a sequence \(\langle A_n : n \in \omega \rangle\), where \(A_n \in \mathscr U_n\) for each \(n\in\omega\),
        so that \(\{ A_n : n \in \omega \} \in \mathcal B\).
    \end{itemize}
    Observe that the selection principles can be concisely captured using quantifiers:
    \[
        \mathsf{S}_{\mathrm{fin}}(\mathcal A, \mathcal B) \equiv
        \left(\forall \mathscr U \in \mathcal A^\omega\right)
        \left( \exists \mathscr F \in \prod_{n\in\omega} \left[ \mathscr U_n \right]^{<\omega} \right)
        \ \bigcup_{n\in\omega} \mathscr F_n \in \mathcal B
    \]
    and
    \[
        \mathsf S_1(\mathcal A, \mathcal B) \equiv
        \left(\forall \mathscr U \in \mathcal A^\omega\right)
        \left( \exists A \in \prod_{n\in\omega} \mathscr U_n \right)
        \ \{A_n : n \in \omega \} \in \mathcal B.
    \]
    We will use the notation \(X \models \mathsf{S}_{\mathrm{fin}}(\mathcal A, \mathcal B)\)
    or \(X \models \mathsf{S}_1(\mathcal A, \mathcal B)\) to denote
    that the corresponding selection principle holds for \(X\).
\end{definition}
Note that \(\mathsf S_{\mathrm{fin}}(\mathcal O,\mathcal O)\) is the \emph{Menger propety}
and that \(\mathsf S_1(\mathcal O, \mathcal O)\) is the \emph{Rothberger property}.
We will call the selection principles \(\mathsf S_{\mathrm{fin}}(\Omega,\Omega)\) and
\(\mathsf S_{\mathrm{fin}}(\mathcal K, \mathcal K)\) the \emph{\(\omega\)-Menger} and \emph{\(k\)-Menger}
properties, respectively.
Likewise, we call the selection principles \(\mathsf S_1(\Omega,\Omega)\) and \(\mathsf S_1(\mathcal K, \mathcal K)\)
the \emph{\(\omega\)-Rothberger} and \emph{\(k\)-Rothberger} properties, respectively.
In any scenario in which we say \(X\) has the \(\mathfrak P\) property, we may also say that
\(X\) is \(\mathfrak P\).
\begin{remark}
Note that the Hurewicz property implies the Menger property.
\end{remark}
As discussed in \cite{COOC7}, the property (\(\ast\)) defined by Gerlits and Nagy \cite{GerlitsNagy}
is shown in \cite{NowikScheepersWeiss} to be equivalent
to the property of being both Hurewicz and Rothberger.\footnote{The results are stated relative to subsets of the real line,
but the proofs lend themselves to greater generality.}
So, following \cite{COOC7}, we take this as our definition.
\begin{definition}
    We say that a space \(X\) is \emph{Gerlits-Nagy} if \(X\) is both Hurewicz and Rothberger.
\end{definition}

The Hurewicz property has a corresponding game.
\begin{definition}
    The \emph{Hurewicz game} on a space \(X\) is played as follows.
    In the \(n^{\mathrm{th}}\) inning, One chooses \(\mathscr U_n \in \mathcal O_X\)
    and Two responds with some \(\mathscr F_n \in [\mathscr U_n]^{<\omega}\).
    Two is declared the winner if
    \[(\forall x \in X)(\exists m \in \omega)(\forall n \geq m)\ x \in \bigcup \mathscr F_n.\]
    Otherwise, One wins.
\end{definition}
Like above, we are avoiding the use of a general game schema, \(\mathsf G_{\mathrm{ufin}}(\mathcal A, \mathcal B)\),
as defined in \cite{ChandraAlam}, in which case the game \(\mathsf G_{\mathrm{ufin}}(\mathcal O, \Gamma)\)
could be seen as the Hurewicz game.
As before, such a formulation would require care, as the authors observe, to account for finite selections which cover the space.

\begin{definition}
    A \emph{strategy} for One in the Hurewicz game on \(X\) is a function
    \[\sigma : \left( \left[\mathcal O_X\right]^{<\omega} \right)^{<\omega} \to \mathcal O_X.\]
    The strategy \(\sigma\) is said to be \emph{winning} if, for any sequence \(\langle \mathscr F_n : n \in \omega \rangle\)
    with \[\mathscr F_n \in \left[ \sigma\left( \left\langle \mathscr F_j : j < n \right\rangle \right) \right]^{<\omega}\]
    for every \(n \in \omega\), then
    \[(\exists x \in X)(\forall m \in \omega)(\exists n \geq m)\ x \not\in \bigcup \mathscr F_n.\]
\end{definition}

Though \cite[Thm. 27]{COOC1} is stated for sets of reals, the proof extends to arbitrary spaces, so we state it
in that level of generality here.
Recall that the main idea of the proof is to use the fact that the Hurewicz property necessitates that a space be
Lindel{\"{o}}f and then to use countable sub-covers of open covers corresponding to a given strategy for One to have Two
respond with initial segments of those countable covers.
Looking at the entire tree of partial plays in this scenario, we can apply the Hurewicz property
on the tree to produce a play which beats the given strategy for One since the Hurewicz property
generates a cover that covers every point by all but finitely members of the selection.
\begin{theorem}[{\cite[Thm. 27]{COOC1}}] \label{thm:HurewiczGameChar}
    A space \(X\) has the Hurewicz property if and only if One does not have a winning strategy in the Hurewicz game
    on \(X\).
\end{theorem}

There are also corresponding games for the standard selection principles \(\mathsf S_{\mathrm{fin}}(\mathcal A, \mathcal B)\)
and \(\mathsf S_1(\mathcal A, \mathcal B)\) that are widely studied in the literature.
We will recall here only the elements of selection game theory relevant to this paper.
\begin{definition}
    Let \(\mathcal A\) and \(\mathcal B\) be collections.
    The \emph{finite-selection game} \(\mathsf G_{\mathrm{fin}}(\mathcal A, \mathcal B)\) is played
    as follows.
    In the \(n^{\mathrm{th}}\) inning, One chooses \(A_n \in \mathcal A\) and Two
    responds with \(\mathscr F_n \in [A_n]^{<\omega}\).
    Two is declared the winner if \(\bigcup_{n\in\omega} \mathscr F_n \in \mathcal B\).
    Otherwise, One wins.
\end{definition}
\begin{definition}
    Let \(\mathcal A\) and \(\mathcal B\) be collections.
    The \emph{single-selection game} \(\mathsf G_1(\mathcal A, \mathcal B)\) is played
    as follows.
    In the \(n^{\mathrm{th}}\) inning, One chooses \(A_n \in \mathcal A\) and Two
    responds with \(x_n \in A_n\).
    Two is declared the winner if \(\{x_n : n \in \omega \} \in \mathcal B\).
    Otherwise, One wins.
\end{definition}
In this paper, we will only be studying two strategy types for One.
\begin{definition}
    We define two strategy types for One.
    \begin{itemize}
    \item
    A \emph{strategy for One} in \(\mathsf{G}_1(\mathcal A, \mathcal B)\) is a function
    \(\sigma:(\bigcup \mathcal A)^{<\omega} \to \mathcal A\).
    A strategy \(\sigma\) for One is called \emph{winning} if whenever \(x_n \in \sigma\langle x_k : k < n \rangle\)
    for all \(n \in \omega\), \(\{x_n: n\in\omega\} \not\in \mathcal B\).
    If One has a winning strategy, we write \(\mathrm{I} \uparrow \mathsf{G}_1(\mathcal A, \mathcal B)\).
    \item
    A \emph{predetermined strategy} for One is a strategy which only considers the current turn number.
    We call this kind of strategy predetermined because One is not reacting to Two's moves,
    they are just running through a script.
    Formally, it is a function \(\sigma: \omega \to \mathcal A\).
    If One has a winning predetermined strategy, we write
    \(\mathrm{I} \underset{\mathrm{pre}}{\uparrow} \mathsf{G}_1(\mathcal A, \mathcal B)\).
    \end{itemize}
    These definitions can be extended to \(\mathsf{G}_{\mathrm{fin}}(\mathcal A, \mathcal B)\) in the obvious way.
\end{definition}
\begin{remark} \label{rmk:SelectionPrinciple}
    In general, \(\textsf{S}_\square(\mathcal A, \mathcal B)\) holds if and only if \(\text{I} \underset{\text{pre}}{\not\uparrow} \textsf{G}_\square(\mathcal{A},\mathcal{B})\) where \(\square \in \{1 , \text{fin} \}\).
    See \cite[Prop. 15]{ClontzDualSelection}.
\end{remark}

\begin{definition}[\cite{COOC7}]
    Let \(\mathfrak C\) be a class.
    We say that \(\mathscr C \in \mathfrak C\) is \emph{groupable} if there exists \(\phi : \mathscr C \to \omega\) so that
    \(\phi^{-1}(n)\) is finite for each \(n \in \omega\) and,
    for any infinite \(J \subseteq \omega\), \(\bigcup \{ \phi^{-1}(n) : n \in J \} \in \mathfrak C\).
    Then, we set \[\mathfrak C^{\mathrm{gp}} = \{ \mathscr C \in \mathfrak C : \mathscr C \text{ is groupable} \}.\]

    If \(\mathfrak O\) is a topological operator, we will let \(\mathfrak O^{\mathrm{gp}}\) be the topological operator
    which generates the groupable elements of \(\mathfrak O_X\) for a space \(X\).
\end{definition}
\begin{remark}
    Note that any groupable element must be countable.
\end{remark}
Originally, the following theorem appeared as a longer list of equivalent conditions, all assuming that \(X\) is
an \(\epsilon\)-space (a space in which every \(\omega\)-cover has a countable subset which is an \(\omega\)-cover).
However, such a hypothesis isn't necessary for the following theorem since both conditions of the equivalence imply
that the space \(X\) is an \(\epsilon\)-space.
\begin{theorem}[{\cite[Thm. 16]{COOC7}}] \label{thm:OmegaGroupChar}
    For a space \(X\), every finite power of \(X\) has the Hurewicz property
    if and only if \(X \models \mathsf S_{\mathrm{fin}}(\Omega , \Omega^{\mathrm{gp}})\).
\end{theorem}

\section{Preparations}

The following is an adaption of \cite[Lemma 9 and Cor. 10]{COOC7} using the same zig-zag technique.
\begin{lemma} \label{lem:CDR}
    For any sequence \(\langle \mathscr U_n : n \in \omega \rangle\) of \(\mathcal O^{\mathrm{gp}}_X(\mathcal A)\),
    there exists a sequence \(\langle \mathscr V_n : n \in \omega \rangle\) of \(\mathcal O^{\mathrm{gp}}_X(\mathcal A)\)
    where the set \(\{ \mathscr V_n : n \in \omega \}\) is pairwise disjoint and
    \(\mathscr V_n \subseteq \mathscr U_n\) for every \(n \in \omega\).
\end{lemma}
\begin{proof}
    Let \(\phi_n : \mathscr U_n \to \omega\) witness the groupability of \(\mathscr U_n\).
    Fix a bijection \(\beta : \omega \to \omega^2\)
    and use the following notation: \(\beta(j) = \langle \beta(j)_1, \beta(j)_2 \rangle\).
    Let \[\mathscr W_{\beta(0)_1, \beta(0)_2} = \mathcal A_0 = \phi^{-1}_{\beta(0)_1}(0).\]

    For \(m \in \omega\), suppose we have \(\mathcal A_m \in \left[\bigcup_{n\in\omega}\mathscr U_n\right]^{<\omega}\)
    and \(\mathscr W_{\beta(m)_1,\beta(m)_2} \subseteq \mathscr U_{\beta(m)_1}\) defined.
    Let \[M = \min \left\{ n \in \omega : (\forall k \geq n) \ \phi^{-1}_{\beta(m+1)_1}(k) \cap \mathcal A_m = \emptyset \right\}.\]
    Such an \(M\) must exist since \(\mathcal A_m\) is finite
    and \(\phi_{\beta(m+1)_1}\) is finite-to-one.
    Then define
    \[\mathscr W_{\beta(m+1)_1,\beta(m+1)_2} = \phi^{-1}_{\beta(m+1)_1}(M)\]
    and \(\mathcal A_{m+1} = \mathcal A_m \cup \mathscr W_{\beta(m+1)_1,\beta(m+1)_2}.\)

    This recursively defines \(\{\mathscr W_{n,m} : n,m \in \omega \}\) where
    \(\mathscr W_{n,m} \subseteq \mathscr U_n\) for every \(n,m \in \omega\).
    Let \[\mathscr V_n = \bigcup \{ \mathscr W_{n,m} : m \in \omega \}\] and observe that
    \(\mathscr V_n \subseteq \mathscr U_n\).
    Also note that \(\mathscr V_n \in \mathcal O_X^{\mathrm{gp}}(\mathcal A)\) for each \(n\in\omega\)
    by the groupability of \(\mathscr U_n\) and that the family
    \(\{ \mathscr V_n : n \in \omega \}\) is pairwise disjoint by construction.
\end{proof}
The following is a generalization of \cite[Lemma 15]{COOC7}.
\begin{lemma} \label{lem:gammaLike}
    Suppose \(X \models \mathsf{S}_{\mathrm{fin}}(\mathcal O_X(\mathcal A),\mathcal O^{\mathrm{gp}}_X(\mathcal A))\).
    Then, for any sequence \(\langle \mathscr U_n : n \in \omega \rangle\)
    of \(\mathcal O_X(\mathcal A)\), there exists a sequence \(\langle \mathscr F_n : n \in \omega \rangle\)
    so that the following properties hold:
    \begin{itemize}
        \item
        \((\forall n \in \omega)\ \mathscr F_n \in \left[ \mathscr U_n \right]^{<\omega}\),
        \item
        \(\{\mathscr F_n : n \in \omega \}\) is pairwise disjoint, and
        \item
        \(\left(\forall A \in \mathcal A \right)(\exists m \in \omega)(\forall n \geq m)(\exists V \in \mathscr F_n)\ A \subseteq V\).
    \end{itemize}
\end{lemma}
\begin{proof}
    By \(X \models \mathsf{S}_{\mathrm{fin}}(\mathcal O_X(\mathcal A),\mathcal O^{\mathrm{gp}}_X(\mathcal A))\)
    we can assume that each \(\mathscr U_n\in\mathcal O^{\mathrm{gp}}_X(\mathcal A)\).
    Then, by Lemma \ref{lem:CDR}, we can assume that \(\{ \mathscr U_n : n \in \omega \}\) is pairwise disjoint.
    Fix a bijective enumeration \(\mathscr U_n = \{ U_{n,m} : m \in \omega \}\) for each \(n \in \omega\)
    and let
    \[\mathscr W_n = \left\{ \bigcap_{j=0}^n U_{j,m_j} : m \in \omega^{n+1}\right\} \setminus \{ \emptyset \}.\]
    Also, for \(n \in \omega\), let \(\gamma_n : \mathscr W_n \to \prod_{j=0}^n \mathscr U_j\) be a choice function\footnote{
    For each \(W \in \mathscr W_n\), the set
    \(\left\{ \langle U_0, U_1, \ldots U_n \rangle \in \prod_{j=0}^n \mathscr U_j : W = \bigcap_{j=0}^n U_j \right\}\)
    is non-empty, so we can use the Axiom of Choice to select an element from each of these sets and assign them to the corresponding
    \(W \in \mathscr W_n\).
    }
    so that,
    for each \(W \in \mathscr W_n\),
    \[W = \bigcap_{j=0}^n \mathrm{proj}_j(\gamma_n(W)),\]
    where \(\mathrm{proj}_j\) is the usual projection mapping.
    Notice that \(\mathscr W_n \in \mathcal O_X(\mathcal A)\) for each \(n \in \omega\)
    so we can apply \(\mathsf{S}_{\mathrm{fin}}(\mathcal O_X(\mathcal A),\mathcal O^{\mathrm{gp}}_X(\mathcal A))\) to produce
    \(\mathscr G_n \in \left[ \mathscr W_n \right]^{<\omega}\)
    so that \[\mathscr W : = \bigcup_{n\in\omega} \mathscr G_n \in \mathcal O^{\mathrm{gp}}_X(\mathcal A).\]
    Let \(\phi : \mathscr W \to \omega\) witness the groupability of \(\mathscr W\) and note that
    \[\left(\forall A \in\mathcal A\right)(\exists m \in \omega)(\forall n \geq m)\left(\exists W \in \phi^{-1}(n)\right)\ A \subseteq W.\]
    Then, for each \(n \in \omega\), let \(\chi_n : \phi^{-1}(n) \to \omega\) be so that,
    for each \(W \in \phi^{-1}(n)\),
    \[\chi_n(W) = \min\{ k \in \omega : W \in \mathscr G_k \}.\]
    Finally, set
    \[\mathscr F_n = \left\{ \mathrm{proj}_n\left(\gamma_{\chi_n(W)}(W)\right)
    : W \in \phi^{-1}(n) \right\} \in \left[ \mathscr U_n \right]^{<\omega}.\]

    Since \(\{ \mathscr U_n : n \in \omega \}\) is pairwise disjoint, the collection \(\{ \mathscr F_n : n \in \omega \}\)
    is also pairwise disjoint.
    To finish the proof, let \(A \in \mathcal A\) be arbitrary and let \(m \in \omega\) be so that,
    for each \(n \geq m\), there is some \(W \in \phi^{-1}(n)\) with \(A \subseteq W\).
    Note that, for \(n \geq m\) and \(W \in \phi^{-1}(n)\) with \(A \subseteq W\),
    we have that
    \[A \subseteq W \subseteq \mathrm{proj}_n(\gamma_{\chi_n(W)}(W)) \in \mathscr F_n.\]
    This finishes the proof.
\end{proof}

\section{The General Hyperspace Correspondence}

We will now restrict our attention to ideals of compact sets;
an \emph{ideal of compact sets} of a space \(X\) is any ideal \(\mathcal A\) of closed sets of \(X\)
so that \(\mathcal A \subseteq K(X)\).
We will also use \([X]^\omega\) to denote the set of countably infinite subsets of a set \(X\).
\begin{definition}
    We say that a pair \((\mathfrak T, \mathfrak A)\) is an \emph{adequate context} if \(\mathfrak T\)
    is a class of topological spaces and \(\mathfrak A\) is a topological operator so that the following properties hold:
    \begin{itemize}
        \item
        For every space \(X \in \mathfrak T\), \(\mathfrak A_X\) is an ideal of compact subsets of \(X\).
        \item
        For every space \(X \in \mathfrak T\), \(\mathbb A(X) \in \mathfrak T\) where \(\mathbb A(X)\)
        is the set \(\mathfrak A_X\) with the topology it inherits as a subspace of \(\mathbb K(X)\).
    \end{itemize}
\end{definition}
As alluded to in the preliminaries, we will be considering two contexts, which will be shown to be adequate below:
the class of non-compact spaces along with the operator that generates all compact subsets,
and the class of all infinite spaces along with the operator that generates all finite subsets.
When the space \(X\) has infinite compact subsets and \(\mathfrak A\) is the operator that generates the finite subsets,
then \(\mathbb A(X) = \mathcal P_{\mathrm{fin}}(X)\) is a proper subset of \(\mathbb K(X)\).

We now offer a modification to \cite[Def. 4.2]{CHVietoris}.
\begin{definition} \label{def:ClosedUnderAUnions}
    Suppose \((\mathfrak T, \mathfrak A)\) is an adequate context.
    We will say that the operator \(\mathfrak A\) is \emph{closed under \(\mathfrak A\)-unions}
    if the following properties hold:
    \begin{itemize}
        \item
        For each \(K_0 \in \mathfrak A_X\), \(\{ K \in \mathbb A(X) : K \subseteq K_0 \} \in \mathfrak A_{\mathbb A(X)}\).
        \item
        For each \(\mathbf K \in \mathfrak A_{\mathbb A(X)}\), \(\bigcup \mathbf K \in \mathfrak A_X\).
    \end{itemize}
\end{definition}
As pointed out by \cite[Lemma 4.3]{CHVietoris}, when \(\mathfrak T\) is all infinite (resp. non-compact) spaces,
the operator \(\mathfrak A\) which produces finite subsets (resp. compact subsets) of a given space \(X\)
is closed under \(\mathfrak A\)-unions.
In the case of the compact subsets, it follows from results of \cite{MichaelTopologies}.
\begin{definition}
    Suppose we have an adequate context \((\mathfrak T, \mathfrak A)\).
    We let \(\mathcal O(\mathfrak A)\) be the topological operator defined by
    \(X \mapsto \mathcal O_X(\mathfrak A_X)\).
    We will say that the operator \(\mathfrak A\) is \emph{suitable} if all of the following properties hold:
    \begin{enumerate}[label=(S\arabic*)]
        \item \label{ClosureUnderA}
        \(\mathfrak A\) is closed under \(\mathfrak A\)-unions.
        \item \label{CoverTransferA}
        For every \(X \in \mathfrak T\) and every \(\mathscr U \in \mathcal O_X(\mathfrak A_X)\),
        \(\{ [U] : U \in \mathscr U \} \in \mathcal O_{\mathbb A(X)}(\mathfrak A_{\mathbb A(X)}).\)
        \item \label{CoverTransferB}
        For every \(X \in \mathfrak T\) and every \(\mathscr U \in \mathcal O_{\mathbb A(X)}(\mathfrak A_{\mathbb A(X)})\),
        \(\{ V \in \mathscr T_X : (\exists U \in \mathscr U)\ [V] \subseteq U \} \in \mathcal O_X(\mathfrak A_X)\).
        \item \label{SufficientForHurewicz}
        For every \(X \in \mathfrak T\), if
        \(X \models \mathsf S_{\mathrm{fin}}(\mathcal O(\mathfrak A), \mathcal O^{\mathrm{gp}}(\mathfrak A))\),
        then \(X\) is Hurewicz.
        \item \label{GeneralMenger}
        For every \(X \in \mathfrak T\), the following are equivalent:
        \begin{itemize}
            \item
            \(\mathbb A(X)\) is Menger.
            \item
            \(\mathbb A(X) \models \mathsf S_{\mathrm{fin}}(\mathcal O(\mathfrak A), \mathcal O(\mathfrak A))\).
            \item
            \(X \models \mathsf S_{\mathrm{fin}}(\mathcal O(\mathfrak A), \mathcal O(\mathfrak A))\).
        \end{itemize}
    \end{enumerate}
\end{definition}
Note that, when \(\mathfrak T\) is the class of all non-compact spaces and \(\mathfrak A\) is the operator
that generates compact subsets, then \(\mathbb A(X) = \mathbb K(X)\) and so the set
\(\mathcal O_{\mathbb A(X)}(\mathfrak A_{\mathbb A(X)})\) is the set of \(k\)-covers of \(\mathbb K(X)\).
Analogously, when \(\mathfrak T\) is the class of all infinite spaces and \(\mathfrak A\) is the operator that
generates finite subsets, then \(\mathcal O_{\mathbb A(X)}(\mathfrak A_{\mathbb A(X)})\) is the set of \(\omega\)-covers of
\(\mathcal P_{\mathrm{fin}}(X)\).
\begin{theorem} \label{thm:GeneralTheorem}
    Suppose we have an adequate context \((\mathfrak T, \mathfrak A)\) where \(\mathfrak A\) is suitable.
    Then, for any \(X \in \mathfrak T\), the following are equivalent.
    \begin{enumerate}[label=(\roman*)]
        \item \label{GeneralHurewicz}
        \(\mathbb A(X)\) is Hurewicz.
        \item \label{GeneralHyperGroup}
        \(\mathbb A(X) \models \mathsf{S}_{\mathrm{fin}}(\mathcal O(\mathfrak A),\mathcal O^{\mathrm{gp}}(\mathfrak A))\).
        \item \label{GeneralGroundGroup}
        \(X \models \mathsf{S}_{\mathrm{fin}}(\mathcal O(\mathfrak A),\mathcal O^{\mathrm{gp}}(\mathfrak A))\).
        \item \label{GeneralGroundCovers}
        \(X \models \mathsf{S}_{\mathrm{fin}}(\mathcal O(\mathfrak A),\mathcal O(\mathfrak A))\)
        and \(\mathcal O_X^{\mathrm{gp}}(\mathfrak A_X) = \mathcal O_X(\mathfrak A_X) \cap [\mathscr T_X]^\omega\).
    \end{enumerate}
\end{theorem}
\begin{proof}
    Note that the implication \ref{GeneralHyperGroup}\(\implies\)\ref{GeneralHurewicz} follows
    from \ref{SufficientForHurewicz}.

    \ref{GeneralHurewicz}\(\implies\)\ref{GeneralGroundCovers}:
    Suppose \(\mathbb A(X)\) is Hurewicz
    and notice that \(\mathbb A(X)\) is Menger,
    which implies that \(X \models \mathsf{S}_{\mathrm{fin}}(\mathcal O(\mathfrak A),\mathcal O(\mathfrak A))\)
    by \ref{GeneralMenger}.
    So we need only show that every countable member of \(\mathcal O_X(\mathfrak A_X)\)
    is groupable.
    So let \(\mathscr U = \{ U_n : n \in \omega \} \in \mathcal O_X(\mathfrak A_X)\).
    We define a strategy for One in the Hurewicz game relative to \(\mathbb A(X)\) in the following way.
    Let
    \[\sigma(\langle \rangle) = \{ [U_n] : n \in \omega \}.\]
    For \(k \in\omega\), suppose we have \(\langle \mathscr G_j : j < k \rangle\), \(\langle M_j : j < k \rangle\),
    and \(\sigma(\langle \mathscr G_j : j < k \rangle)\) defined.
    For any \[\mathscr G_k \in \left[ \sigma(\langle \mathscr G_j : j < k \rangle) \right]^{<\omega},\]
    let \[M_k = \max\{ m \in \omega : [U_m] \in \mathscr G_k \}.\]
    Then set
    \[\sigma(\langle \mathscr G_j : j \leq k \rangle) =
    \left\{ [U_n] : n > M_k \right\}.\]
    Notice that, by \ref{CoverTransferA},
    \(\sigma(\langle \mathscr G_j : j \leq k \rangle) \in \mathcal O_{\mathbb A(X)}(\mathfrak A_{\mathbb A(X)})\) since
    \(\{ U_n : n > M_k \} \in \mathcal O_X(\mathfrak A_X)\).
    So this \(\sigma\) is a strategy for One in the Hurewicz game on \(\mathbb A(X)\).

    By Theorem \ref{thm:HurewiczGameChar}, \(\sigma\) cannot be a winning strategy.
    So there is a sequence \(\langle \mathscr G_j : j \in \omega \rangle\) of finite selections and
    a corresponding increasing sequence \(\langle M_j : j \in \omega \rangle\) of naturals as defined
    above that win against \(\sigma\).
    We can then define \(\phi : \mathscr U \to \omega\) by the rule
    \[\phi(U_n) = \min \{ k \in \omega : n \leq M_k \}.\]
    Since \(\phi\) is evidently finite-to-one, we need only show it has the required
    covering property.
    So let \(J \subseteq \omega\) be infinite and \(K \in \mathfrak A_X\).
    By our construction, there must be some \(\ell \in J\), \(\ell > 0\), and some
    \(W \in \mathscr G_\ell\) for which \(K \in W\).
    Hence, there must be some \(n \in (M_{\ell-1}, M_\ell]\) so that \(W = [U_n]\).
    Note that \(\phi(U_n) = \ell\) and that \(K \subseteq U_n\).
    That is, \(K \subseteq U_n \in \phi^{-1}(\ell)\), which establishes that \(\{ U_n : n \in \omega \}\)
    is groupable.

    The implication \ref{GeneralGroundCovers}\(\implies\)\ref{GeneralGroundGroup} is evident.

    \ref{GeneralGroundGroup}\(\implies\)\ref{GeneralHurewicz}:
    Suppose \(X \models \mathsf{S}_{\mathrm{fin}}(\mathcal O(\mathfrak A),\mathcal O^{\mathrm{gp}}(\mathfrak A))\)
    and let \(\langle \mathscr U_n : n \in \omega \rangle\)
    be a sequence of open covers of \(\mathbb A(X)\).
    Note that \[\mathscr U_n^\ast := \left\{ \bigcup \mathscr F : \mathscr F \in [\mathscr U_n]^{<\omega} \right\}
    \in \mathcal O_{\mathbb A(X)}(\mathfrak A_{\mathbb A(X)})\]
    since \(\mathfrak A_{\mathbb A(X)}\) consists of compact sets.
    Fix a choice function \(\gamma_{u,n} : \mathscr U_n^\ast \to [\mathscr U_n]^{<\omega}\) to be so that
    \(\bigcup \gamma_{u,n}(U) = U\).
    By \ref{CoverTransferB},
    \[\mathscr V_n := \{ V \in \mathscr T_X : (\exists U \in \mathscr U_n^\ast)\ [V] \subseteq U \} \in \mathcal O_X(\mathfrak A_X).\]
    Fix a choice function \(\gamma_{v,n} : \mathscr V_n \to \mathscr U^\ast_n\) so that
    \([V] \subseteq \gamma_{v,n}(V)\).
    By Lemma \ref{lem:gammaLike}, we can choose \(\mathscr G_n \in [\mathscr V_n]^{<\omega}\) with the properties
    guaranteed by the lemma.
    Define \[\mathscr F_n = \bigcup \{ \gamma_{u,n} \circ \gamma_{v,n}(V) : V \in \mathscr G_n \}\]
    and note that \(\mathscr F_n \in [\mathscr U_n]^{<\omega}\).

    We now show that the selection \(\{\mathscr F_n: n \in \omega\}\) witnesses the Hurewicz property for
    \(\mathbb A(X)\).
    So let \(K \in \mathbb A(X)\).
    By our use of Lemma \ref{lem:gammaLike}, we can find \(m \in \omega\) so that, for every \(n \geq m\),
    there is some \(V \in \mathscr G_n\) with
    \[K \subseteq V \implies K \in [V] \subseteq \gamma_{v,n}(V) = \bigcup \gamma_{u,n} \circ \gamma_{v,n}(V) \subseteq \bigcup \mathscr F_n.\]
    Hence, \(\mathbb A(X)\) is Hurewicz.

    \ref{GeneralHurewicz}\(\implies\)\ref{GeneralHyperGroup}:
    Suppose \(\mathbb A(X)\) is Hurewicz.
    By the equivalence of \ref{GeneralHurewicz} and \ref{GeneralGroundGroup} established above,
    we know that \(X \models \mathsf{S}_{\mathrm{fin}}(\mathcal O(\mathfrak A),\mathcal O^{\mathrm{gp}}(\mathfrak A))\).
    It follows that
    \(\mathbb A(X) \models \mathsf{S}_{\mathrm{fin}}(\mathcal O(\mathfrak A),\mathcal O(\mathfrak A))\)
    by \ref{GeneralMenger}.
    So we need only show every countable member of \(\mathcal O_{\mathbb A(X)}(\mathfrak A_{\mathbb A(X)})\) is groupable.
    So let \(\mathscr U = \{ U_n : n \in \omega \} \in \mathcal O_{\mathbb A(X)}(\mathfrak A_{\mathbb A(X)})\).
    We define a strategy for One in the Hurewicz game on \(\mathbb A(X)\).

    Throughout this portion of the proof, we will be using the fact that
    \(X \models \mathsf{S}_{\mathrm{fin}}(\mathcal O(\mathfrak A),\mathcal O^{\mathrm{gp}}(\mathfrak A))\)
    to find countable subsets of members of \(\mathcal O_X(\mathfrak A_X)\) which are in \(\mathcal O_X(\mathfrak A_X)\).
    We will also use \ref{CoverTransferB} to take members of \(\mathcal O_{\mathbb A(X)}(\mathfrak A_{\mathbb A(X)})\)
    and refine them with sets of the form \([V]\) for \(V \in\mathscr T_X\) that form elements
    of \(\mathcal O_X(\mathfrak A_X)\).

    To begin, let \(\{ V_m : m \in \omega \} \in \mathcal O_X(\mathfrak A_X)\) be so that, for every \(m \in \omega\),
    there is \(n \in \omega\) with \([V_m] \subseteq U_n\).
    Then we let \[\sigma(\langle \rangle) = \mathscr V_{\langle\rangle} := \{ [V_m] : m \in \omega \}.\]
    For \(\mathscr G_0 \in [\mathscr V_{\langle \rangle}]^{<\omega}\), define
    \[M_0 = \min\{ N \in \omega : (\forall W \in \mathscr G_0)(\exists n \leq N)\ W \subseteq U_n \}.\]
    Note that \(\{ U_n : n > M_0 \} \in \mathcal O_{\mathbb A(X)}(\mathfrak A_{\mathbb A(X)})\).
    So we can let \(\{ V_{0 , m} : m \in \omega \} \in \mathcal O_X(\mathfrak A_X)\)
    be so that, for every \(m \in \omega\), there is some \(n > M_0\) with \([V_{0,m}] \subseteq U_n\).
    Define
    \[\sigma(\langle \mathscr G_0 \rangle) = \mathscr V_{\langle M_0 \rangle}
    := \{ [V_{0,m}] : m \in \omega \}.\]

    Now, let \(k \in \omega\) and suppose we have \(\mathscr G_j\), \(M_j\), for \(j \leq k\), and
    \(\mathscr V_{\langle M_j : j \leq k\rangle}\) defined where \(\mathscr V_{\langle M_j : j \leq k\rangle}\)
    refines \(\mathscr U\).
    For \(\mathscr G_{k+1} \in [\mathscr V_{\langle M_j : j \leq k \rangle}]^{<\omega}\), let
    \[M_{k+1} = \min\{ N > M_k : (\forall W \in \mathscr G_{k+1})(\exists n \leq N)\ W \subseteq U_n \}.\]
    Let \(\{ V_{k+1,m} : m \in \omega \} \in \mathcal O_X(\mathfrak A_X)\) be so that,
    for every \(m \in \omega\), there is some \(n > M_{k+1}\)
    so that \([V_{k+1,m}] \subseteq U_n.\)
    Define
    \[\sigma(\langle\mathscr G_j : j \leq k+1 \rangle) = \mathscr V_{\langle M_j : j \leq k+1 \rangle}
    := \{ [V_{k+1,m}] : m \in \omega \}.\]
    This defines a strategy \(\sigma\) for One in the Hurewicz game on \(\mathbb A(X)\).

    By Theorem \ref{thm:HurewiczGameChar},
    we know that \(\sigma\) cannot be winning.
    So there is a sequence \(\langle \mathscr G_j : j \in \omega \rangle\) and a corresponding increasing
    sequence \(\langle M_j : j \in \omega \rangle\) of naturals so that, for any \(K \in \mathbb A(X)\),
    there exists some \(m \in \omega\) so that, for every \(n \geq m\), \(K \in \bigcup \mathscr G_n\).
    We can define \(\phi : \mathscr U \to \omega\) by the rule
    \[\phi(U_n) = \min\{ k \in \omega : n \leq M_k \}.\]
    Note that \(\phi\) is clearly finite-to-one.
    So we need only show it satisfies the covering criterion for groupability of
    \(\mathcal O_{\mathbb A(X)}(\mathfrak A_{\mathbb A(X)})\).
    Let \(J \subseteq \omega\) be infinite and \(\mathbf K \in \mathfrak A_{\mathbb A(X)}\).
    Observe that \(\bigcup \mathbf K \in \mathfrak A_X\) by \ref{ClosureUnderA}.
    There must be some \(\ell \in J\), \(\ell > 0\), and some \(W \in \mathscr T_X\) with
    \(\bigcup \mathbf K \in [W] \in \mathscr G_\ell\).
    Note that \(\bigcup \mathbf K \in [W]\) means that \(\bigcup \mathbf K \subseteq W\).
    By construction, there is some \(n \in (M_{\ell-1},M_\ell]\) so that \([W] \subseteq U_n\).
    Hence, for any \(K \in \mathbf K\), \(K \subseteq \bigcup \mathbf K \subseteq W.\)
    That is, \(K \in [W] \subseteq U_n\).
    Since \(K \in \mathbf K\) was arbitrary, \(\mathbf K \subseteq U_n \in \phi^{-1}(\ell)\).
    This establishes that \(\mathscr U\) is groupable, and thus, that
    \(\mathbb A(X) \models \mathsf S_{\mathrm{fin}}(\mathcal O(\mathfrak A),\mathcal O^{\mathrm{gp}}(\mathfrak A))\).
\end{proof}
We now offer a game-theoretic characterization.
\begin{theorem} \label{thm:GeneralStrategic}
    Let \(\square \in \{1 , \mathrm{fin} \}\).
    Suppose we have an adequate context \((\mathfrak T, \mathfrak A)\) where \(\mathfrak A\) is suitable and,
    for every \(X \in \mathfrak T\),
    \[\mathrm{I} \underset{\mathrm{pre}}{\uparrow} \mathsf G_{\square}(\mathcal O_X(\mathfrak A_X),\mathcal O_X(\mathfrak A_X))
    \iff \mathrm{I} \uparrow \mathsf G_{\square}(\mathcal O_X(\mathfrak A_X),\mathcal O_X(\mathfrak A_X)).\]
    Then, for any \(X \in \mathfrak T\),
    \[\mathrm{I} \underset{\mathrm{pre}}{\uparrow} \mathsf G_{\square}(\mathcal O_X(\mathfrak A_X),\mathcal O_X^{\mathrm{gp}}(\mathfrak A_X))
    \iff \mathrm{I} \uparrow \mathsf G_{\square}(\mathcal O_X(\mathfrak A_X),\mathcal O_X^{\mathrm{gp}}(\mathfrak A_X)).\]
\end{theorem}
\begin{proof}
    We need only show that, for \(X \in \mathfrak T\),
    \[\mathrm{I} \underset{\mathrm{pre}}{\not\uparrow} \mathsf G_{\square}(\mathcal O_X(\mathfrak A_X),\mathcal O_X^{\mathrm{gp}}(\mathfrak A_X))
    \implies \mathrm{I} \not\uparrow \mathsf G_{\square}(\mathcal O_X(\mathfrak A_X),\mathcal O_X^{\mathrm{gp}}(\mathfrak A_X)).\]
    So suppose \(\mathrm{I} \underset{\mathrm{pre}}{\not\uparrow} \mathsf G_{\square}(\mathcal O_X(\mathfrak A_X),\mathcal O_X^{\mathrm{gp}}(\mathfrak A_X))\).
    By Remark \ref{rmk:SelectionPrinciple} and Theorem \ref{thm:GeneralTheorem},
    \(X \models \mathsf S_\square(\mathcal O(\mathfrak A),\mathcal O(\mathfrak A))\) and
    \(\mathcal O_X^{\mathrm{gp}}(\mathfrak A_X) = \mathcal O_X(\mathfrak A_X) \cap [\mathscr T_X]^\omega\).
    By the hypothesis, we see that
    \(\mathrm{I} \not\uparrow \mathsf G_\square(\mathcal O_X(\mathfrak A_X), \mathcal O_X(\mathfrak A_X))\).
    Now, since any strategy \(\sigma\) employed by One in the game cannot be winning, Two can always beat
    \(\sigma\) and, since every countable member of \(\mathcal O_X(\mathfrak A_X)\) is groupable,
    Two's winning play is always groupable.
    Hence, One does not have a winning strategy in
    \(\mathsf G_\square(\mathcal O_X(\mathfrak A_X), \mathcal O^{\mathrm{gp}}_X(\mathfrak A_X))\).
\end{proof}
To address single-selections, we need a bit more in our hypotheses.
\begin{definition}
    Suppose we have an adequate context \((\mathfrak T, \mathfrak A)\).
    We will say that the operator \(\mathfrak A\) is \emph{super suitable} if it is suitable and,
    for all \(X \in \mathfrak T\), the following are equivalent:
    \begin{itemize}
        \item
        \(\mathbb A(X)\) is Rothberger.
        \item
        \(\mathbb A(X) \models \mathsf{S}_{1}(\mathcal O(\mathfrak A),\mathcal O(\mathfrak A))\).
        \item
        \(X \models \mathsf{S}_{1}(\mathcal O(\mathfrak A),\mathcal O(\mathfrak A))\).
    \end{itemize}
\end{definition}
\begin{theorem} \label{thm:GeneralSingleSelections}
    Suppose we have an adequate context \((\mathfrak T, \mathfrak A)\) where \(\mathfrak A\) is super suitable.
    Then, for any \(X \in \mathfrak T\), the following are equivalent.
    \begin{enumerate}[label=(\roman*)]
        \item \label{GeneralGerlitsNagy}
        \(\mathbb A(X)\) is Gerlits-Nagy.
        \item \label{GeneralSingleHyper}
        \(\mathbb A(X) \models \mathsf{S}_{1}(\mathcal O(\mathfrak A),\mathcal O^{\mathrm{gp}}(\mathfrak A))\).
        \item \label{GeneralSingleGround}
        \(X \models \mathsf{S}_{1}(\mathcal O(\mathfrak A),\mathcal O^{\mathrm{gp}}(\mathfrak A))\).
    \end{enumerate}
\end{theorem}
\begin{proof}
    \ref{GeneralGerlitsNagy}\(\implies\)\ref{GeneralSingleGround}:
    Since \(\mathbb A(X)\) is Hurewicz, Theorem \ref{thm:GeneralTheorem} asserts that
    \(\mathcal O_X^{\mathrm{gp}}(\mathfrak A_X) = \mathcal O_X(\mathfrak A_X) \cap [\mathscr T_X]^\omega\).
    Since \(\mathbb A(X)\) is Rothberger, by the criterion in the definition of super suitable,
    \(X \models \mathsf S_1(\mathcal O(\mathfrak A), \mathcal O(\mathfrak A))\).
    Hence,
    \(X \models \mathsf S_1(\mathcal O(\mathfrak A), \mathcal O^{\mathrm{gp}}(\mathfrak A))\).

    \ref{GeneralSingleGround}\(\implies\)\ref{GeneralSingleHyper}:
    By Theorem \ref{thm:GeneralTheorem},
    \(\mathbb A(X) \models \mathsf{S}_{\mathrm{fin}}(\mathcal O(\mathfrak A),\mathcal O^{\mathrm{gp}}(\mathfrak A))\).
    Moreover, also by an application of Theorem \ref{thm:GeneralTheorem},
    \[\mathcal O_{\mathbb A(X)}^{\mathrm{gp}}(\mathfrak A_{\mathbb A(X)})
    = \mathcal O_{\mathbb A(X)}(\mathfrak A_{\mathbb A(X)}) \cap [\mathscr T_{\mathbb A(X)}]^\omega.\]
    By the criterion in the definition of super suitable,
    \(\mathbb A(X) \models \mathsf{S}_{1}(\mathcal O(\mathfrak A),\mathcal O(\mathfrak A))\), and, by the above identity,
    we see that \(\mathbb A(X) \models \mathsf{S}_{1}(\mathcal O(\mathfrak A),\mathcal O^{\mathrm{gp}}(\mathfrak A))\).

    \ref{GeneralSingleHyper}\(\implies\)\ref{GeneralGerlitsNagy}:
    Since \(\mathbb A(X) \models \mathsf{S}_{1}(\mathcal O(\mathfrak A),\mathcal O^{\mathrm{gp}}(\mathfrak A)),\)
    \(\mathbb A(X) \models \mathsf{S}_{\mathrm{fin}}(\mathcal O(\mathfrak A),\mathcal O^{\mathrm{gp}}(\mathfrak A))\) which
    establishes that \(\mathbb A(X)\) is Hurewicz by Theorem \ref{thm:GeneralTheorem}.
    Also, by the criterion in the definition of super suitable, \(\mathbb A(X)\) is Rothberger.
\end{proof}

\section{Applications of the Hyperspace Correspondence}

\subsection{The Compact Sets}

It was shown in \cite{CHVietoris} that \(\mathbb K(X)\) captures the properties
of \(k\)-Lindel{\"{o}}f and \(k\)-Menger on \(X\).
Notably, \(k\)-Rothberger isn't in this list.
Indeed, \cite[Ex. 4.19]{CHVietoris} points out that \(\mathbb R\) is \(k\)-Rothberger but not Rothberger.
Since the property of Rothberger is hereditary with respect to closed subsets and \(\mathbb R\) embeds as
a closed subspace of \(\mathbb K(\mathbb R)\), \(\mathbb K(\mathbb R)\) is not Rothberger.
However, Theorem \ref{thm:GeneralTheorem} can be used to establish a relationship between \(X\) and \(\mathbb K(X)\) relative to the Hurewicz property.

We start by recalling the following fact regarding \(k\)-covers.
\begin{lemma}[{\cite[Cor. 4.13]{CHVietoris}}] \label{lem:kCoverThing}
    Let \(X\) be a space.
    \begin{itemize}
        \item
        If \(\mathscr U\) is a \(k\)-cover of \(X\), then \(\{ [U] : U \in \mathscr U \}\)
        is a \(k\)-cover of \(\mathbb K(X)\).
        \item
        If \(\mathscr U\) is a \(k\)-cover of \(\mathbb K(X)\), then
        \[\mathscr V := \{ V \in \mathscr T_X : (\exists U \in \mathscr U)\ [V] \subseteq U \}\]
        is a \(k\)-cover of \(X\).
    \end{itemize}
\end{lemma}
\begin{theorem}[{\cite[Cor. 4.17]{CHVietoris}}] \label{thm:KMenger}
    For any space \(X\), the following are equivalent:
    \begin{itemize}
        \item
        \(\mathbb K(X)\) is Menger.
        \item
        \(\mathbb K(X)\) is \(k\)-Menger.
        \item
        \(X\) is \(k\)-Menger.
    \end{itemize}
\end{theorem}
\begin{lemma} \label{lem:MainKLemma}
    For any space \(X\),
    if \(X \models \mathsf S_{\mathrm{fin}}(\mathcal K, \mathcal K^{\mathrm{gp}})\),
    then \(X\) is Hurewicz.
\end{lemma}
\begin{proof}
    Suppose \(X \models \mathsf S_{\mathrm{fin}}(\mathcal K, \mathcal K^{\mathrm{gp}})\)
    and let \(\langle \mathscr U_n : n \in \omega \rangle\) be a sequence of open covers of \(X\).
    Note that
    \[\mathscr V_n := \left\{ \bigcup \mathscr F : \mathscr F \in [\mathscr U_n]^{<\omega} \right\} \in \mathcal K_X.\]
    By Lemma \ref{lem:gammaLike}, we can produce
    \(\mathscr G_n \in \left[ \mathscr V_n \right]^{<\omega}\) for each \(n \in \omega\)
    with the properties guaranteed in the lemma.
    Let \(\gamma_n : \mathscr G_n \to \left[ \mathscr U_n \right]^{<\omega}\) be a choice function
    so that \(V = \bigcup \gamma_n(V)\) for each \(V\in \mathscr G_n\).
    Then define
    \[\mathscr F_n = \bigcup_{V\in\mathscr G_n} \gamma_n(V)\]
    and notice that \(\mathscr F_n \in \left[ \mathscr U_n \right]^{<\omega}\).
    To finish the proof, we need only show that \(\langle \mathscr F_n : n \in \omega \rangle\)
    witnesses the Hurewicz criterion.
    So let \(x \in X\) and observe that \(\{x\} \in \mathbb K(X)\).
    By the use of Lemma \ref{lem:gammaLike}, there is some \(m \in \omega\) so that,
    for every \(n \geq m\), \(\{x\} \subseteq \bigcup \mathscr G_n\).
    Hence, for every \(n \geq m\), there is some \(V \in \mathscr G_n\) for which \(x \in V\).
    Observe that, for \(V \in \mathscr G_n\), \[V = \bigcup \gamma_n(V)
    \subseteq \bigcup_{V \in \mathscr G_n} \bigcup \gamma_n(V) = \bigcup \mathscr F_n.\]
    Therefore, for every \(n \geq m\), \(x \in \bigcup \mathscr F_n\).
\end{proof}
We will also use the following game-theoretic results.
\begin{theorem}[{\cite[Cor. 4.18 and Thm. 4.21]{CHVietoris}}] \label{kPawlikowski}
    For any space \(X\) and \(\square \in \{1 , \mathrm{fin} \}\),
    \[\mathrm{I} \underset{\mathrm{pre}}{\uparrow} \mathsf G_{\square}(\mathcal K_X,\mathcal K_X)
    \iff \mathrm{I} \uparrow \mathsf G_{\square}(\mathcal K_X,\mathcal K_X).\]
\end{theorem}
\begin{corollary} \label{cor:MainTheorem}
    For any space \(X\), the following are equivalent:
    \begin{enumerate}[label=(\roman*)]
        \item \label{thm:CompactHurewicz}
        \(\mathbb K(X)\) is Hurewicz.
        \item \label{thm:CompactHyperK}
        \(\mathbb K(X) \models \mathsf{S}_{\mathrm{fin}}(\mathcal K,\mathcal K^{\mathrm{gp}})\).
        \item
        \(X \models \mathsf S_{\mathrm{fin}}(\mathcal K, \mathcal K)\) and \(\mathcal K_X^{\mathrm{gp}} = \mathcal K_X \cap [\mathscr T_X]^\omega\).
        \item \label{thm:CompatK}
        \(X \models \mathsf{S}_{\mathrm{fin}}(\mathcal K,\mathcal K^{\mathrm{gp}})\).
        \item \label{kFiniteGameTheoretic}
        One does not have a winning strategy in \(\mathsf G_{\mathrm{fin}}(\mathcal K_X, \mathcal K_X^{\mathrm{gp}})\).
    \end{enumerate}
\end{corollary}
\begin{proof}
    Let \(\mathfrak T\) be the class of all non-compact spaces and \(\mathfrak A\) be the operator
    that generates the compact subsets of a given space \(X\).
    Then \((\mathfrak T, \mathfrak A)\) is an adequate context since it is well-known that \(X\) is compact
    if and only if \(\mathbb K(X)\) is compact \cite{MichaelTopologies}.
    As mentioned above, by results of \cite{MichaelTopologies}, \(\mathfrak A\) is closed under \(\mathfrak A\)-unions.
    Lemmas \ref{lem:kCoverThing} and \ref{lem:MainKLemma} and Theorem \ref{thm:KMenger} then
    finish showing that \(\mathfrak A\) is suitable.
    Thus, Theorem \ref{thm:GeneralTheorem} applies to give the equivalence of \ref{thm:CompactHurewicz}--\ref{thm:CompatK}.
    The equivalence of \ref{thm:CompatK} and \ref{kFiniteGameTheoretic}
    follows from Theorems \ref{thm:GeneralStrategic} and \ref{kPawlikowski}.
\end{proof}
As mentioned above, \(\mathbb R\) is an example of a space which is \(k\)-Rothberger
but where \(\mathbb K(X)\) is not Rothberger.
So, we cannot apply Theorem \ref{thm:GeneralSingleSelections} in this setting.
In fact, we can even say more.
\begin{example} \label{example:RealsAgain}
    The real line \(\mathbb R\) has the property that
    \(\mathbb R \models \mathsf S_1(\mathcal K, \mathcal K^{\mathrm{gp}})\)
    but \(\mathbb K(\mathbb R)\) is not Gerlits-Nagy.

    Indeed, for any sequence \(\langle \mathscr U_n : n \in \omega \rangle\)
    of \(k\)-covers of \(\mathbb R\), choose \(U_n \in \mathscr U_n \setminus \{ U_k : k < n \}\)
    to be so that \([-2^n,2^n] \subseteq U_n\).
    Then define \(\phi : \{ U_n : n \in \omega \} \to \omega\) by \(\phi(U_n) = n\).
    Thus, \(\{ U_n : n \in \omega \}\) is groupable.

    On the other hand, as noted previously, \(\mathbb K(\mathbb R)\) is not Rothberger,
    so \(\mathbb K(\mathbb R)\) is not Gerlits-Nagy.
\end{example}
Therefore, in light of Example \ref{example:RealsAgain}, we see that a single-selection
analog to Corollary \ref{cor:MainTheorem} is not obtained.
Nevertheless, we still obtain the following as a direct application of Theorems \ref{thm:GeneralStrategic}
and \ref{kPawlikowski}.
\begin{corollary} \label{cor:GroupPawlikowski}
    For any space \(X\),
    \[\mathrm{I} \underset{\mathrm{pre}}{\uparrow} \mathsf G_1(\mathcal K_X,\mathcal K_X^{\mathrm{gp}})
    \iff \mathrm{I} \uparrow \mathsf G_1(\mathcal K_X,\mathcal K_X^{\mathrm{gp}}).\]
\end{corollary}

\subsection{The Finite Sets}

In \cite{CHVietoris}, it was shown that \(\mathcal P_{\mathrm{fin}}(X)\) captures
the properties of \(\omega\)-Lindel{\"{o}}f, \(\omega\)-Menger, and \(\omega\)-Rothberger on \(X\),
just as the finite powers of \(X\) do (see \cite{GerlitsNagy,JustMillerScheepers,SakaiCDoublePrime}).
We use Theorem \ref{thm:GeneralTheorem} to extend this relationship to the Hurewicz property.

The following can be seen as a sort of analog of The Wallace Theorem \cite[Thm. 3.2.10]{Engelking}
in the context of finite subsets instead of finite powers.
\begin{lemma}[{\cite[Cor. 4.7]{CHVietoris}}] \label{lem:OmegaCoverThing}
    Let \(X\) be a space.
    \begin{itemize}
        \item
        If \(\mathscr U\) is an \(\omega\)-cover of \(X\), then \(\{ [U] : U \in \mathscr U \}\)
        is an \(\omega\)-cover of \(\mathcal P_{\mathrm{fin}}(X)\).
        \item
        If \(\mathscr U\) is an \(\omega\)-cover of \(\mathcal P_{\mathrm{fin}}(X)\), then
        \[\mathscr V := \{ V \in \mathscr T_X : (\exists U \in \mathscr U)\ [V] \subseteq U \}\]
        is an \(\omega\)-cover of \(X\).
    \end{itemize}
\end{lemma}

\begin{theorem}[{\cite[Cor. 4.10]{CHVietoris}}] \label{thm:OmegaMenger}
    For any space \(X\), the following are equivalent:
    \begin{itemize}
        \item
        \(\mathcal P_{\mathrm{fin}}(X)\) is Menger.
        \item
        \(\mathcal P_{\mathrm{fin}}(X)\) is \(\omega\)-Menger.
        \item
        \(X\) is \(\omega\)-Menger.
    \end{itemize}
\end{theorem}
\begin{theorem}[\cite{Scheepers3}] \label{ScheepersOmegaGame}
    For any space \(X\) and \(\square\in\{1,\mathrm{fin}\}\),
    \[
        \mathrm{I} \underset{\mathrm{pre}}{\uparrow} \mathsf{G}_{\square}(\Omega_X,\Omega_X)
        \iff \mathrm{I} \uparrow \mathsf{G}_{\square}(\Omega_X,\Omega_X).
    \]
\end{theorem}
\begin{corollary} \label{cor:MainTheoremFinite}
    For any space \(X\), the following are equivalent:
    \begin{enumerate}[label=(\roman*)]
        \item \label{thm:FiniteHurewicz}
        \(\mathcal P_{\mathrm{fin}}(X)\) is Hurewicz.
        \item \label{thm:FiniteHyperOmega}
        \(\mathcal P_{\mathrm{fin}}(X) \models \mathsf{S}_{\mathrm{fin}}(\Omega,\Omega^{\mathrm{gp}})\).
        \item \label{thm:FiniteOmega}
        \(X \models \mathsf{S}_{\mathrm{fin}}(\Omega,\Omega^{\mathrm{gp}})\).
        \item \label{EveryFinitePower}
        Every finite power of \(X\) is Hurewicz.
        \item \label{HurewiczAndOmegaGame}
        One does not have a winning strategy in \(\mathsf G_{\mathrm{fin}}(\Omega_X,\Omega_X^{\mathrm{gp}})\).
    \end{enumerate}
\end{corollary}
\begin{proof}
    Let \(\mathfrak T\) be the class of all infinite spaces and \(\mathfrak A\) be the operator
    that generates the finite subsets of a given space \(X\).
    Then \((\mathfrak T, \mathfrak A)\) is an adequate context.
    As mentioned above, \(\mathfrak A\) is closed under \(\mathfrak A\)-unions.
    To finish showing that \(\mathfrak A\) is suitable,
    use Lemma \ref{lem:OmegaCoverThing} and Theorems \ref{thm:OmegaGroupChar} and \ref{thm:OmegaMenger}.
    Thus, Theorem \ref{thm:GeneralTheorem} applies to give us the equivalence of
    \ref{thm:FiniteHurewicz}--\ref{thm:FiniteOmega}.
    The equivalence of \ref{thm:FiniteOmega} and \ref{EveryFinitePower} is the content of
    Theorem \ref{thm:OmegaGroupChar}.
    The equivalence of \ref{thm:FiniteOmega} and \ref{HurewiczAndOmegaGame} follows from Theorems \ref{thm:GeneralStrategic}
    and \ref{ScheepersOmegaGame}.
\end{proof}
In this case, we can appeal to Theorem \ref{thm:GeneralSingleSelections} since we have access to the following result.
\begin{theorem}[{\cite[Cor. 4.11]{CHVietoris}}] \label{thm:OmegaRothberger}
    For any space \(X\), the following are equivalent:
    \begin{itemize}
        \item
        \(\mathcal P_{\mathrm{fin}}(X)\) is Rothberger.
        \item
        \(\mathcal P_{\mathrm{fin}}(X)\) is \(\omega\)-Rothberger.
        \item
        \(X\) is \(\omega\)-Rothberger.
    \end{itemize}
\end{theorem}
\begin{corollary} \label{cor:MainFiniteSingle}
    For any space \(X\), the following are equivalent:
    \begin{enumerate}[label=(\roman*)]
        \item \label{GerlitsNagyBegin}
        \(\mathcal P_{\mathrm{fin}}(X)\) is Gerlits-Nagy.
        \item
        \(\mathcal P_{\mathrm{fin}}(X) \models \mathsf{S}_{1}(\Omega,\Omega^{\mathrm{gp}})\).
        \item \label{GerlitsNagyEnd}
        \(X \models \mathsf{S}_{1}(\Omega,\Omega^{\mathrm{gp}})\).
        \item \label{GerlitsNagyFinitePowers}
        Every finite power of \(X\) is Gerlits-Nagy.
        \item \label{GerlitsNagyOmega}
        One does not have a winning strategy in \(\mathsf G_1(\Omega_X,\Omega_X^{\mathrm{gp}})\).
    \end{enumerate}
\end{corollary}
\begin{proof}
    Let \(\mathfrak T\) be the class of infinite spaces and \(\mathfrak A\)
    be the operator that produces the finite subsets of a given space \(X\).
    It was shown above that \(\mathfrak A\) is suitable.
    Theorem \ref{thm:OmegaRothberger} establishes that \(\mathfrak A\) is super suitable,
    so Theorem \ref{thm:GeneralSingleSelections} applies to provide the equivalence of
    \ref{GerlitsNagyBegin}--\ref{GerlitsNagyEnd}.
    The equivalence of \ref{GerlitsNagyEnd} and \ref{GerlitsNagyFinitePowers} is established
    in \cite[Thm. 19]{COOC7}.
    The equivalence of \ref{GerlitsNagyEnd} and \ref{GerlitsNagyOmega} follows from Theorems \ref{thm:GeneralStrategic}
    and \ref{ScheepersOmegaGame}.
\end{proof}

We end this section with a question.
\begin{question}
    Are there other natural examples of (super) suitable operators \(\mathfrak A\) other than those listed
above?
\end{question}

\section{Relationships with Spaces of Continuous Functions}

For a space \(X\), let \(C(X)\) denote the set of all continuous real-valued function on \(X\).
Given an ideal of closed sets \(\mathcal A\) of \(X\),
let \(C_{\mathcal A}(X)\) denote the set \(C(X)\) endowed with the topology of uniform convergence of
elements of \(\mathcal A\);
this topology has as a basis sets of the form
\[[f; A, \varepsilon] := \left\{ g \in C(X) : \sup \{ |f(x) - g(x)| : x \in A \} < \varepsilon \right\}\]
for \(f \in C(X)\), \(A \in \mathcal A\), and \(\varepsilon > 0\).
When \(\mathcal A = [X]^{<\omega}\) (resp. \(\mathcal A = K(X)\)), we use \(C_p(X)\)
(resp. \(C_k(X)\)) instead of \(C_{\mathcal A}(X)\).
Indeed, \(C_p(X)\) coincides with the ring of continuous real-valuded functions
on \(X\) with the topology of point-wise convergence.
Likewise, \(C_k(X)\) coincides with the ring of continuous real-valuded functions
on \(X\) with the topology of uniform convergence on compacta, which is equivalent to the
compact-open topology.

Arhangel’skii, in \cite{ArkhangelskiiHurewicz}, introduced the notion of countable fan-tightness
and gave a characterization of a space \(X\) being Hurewicz in all of its finite powers
in terms of \(C_p(X)\).
\begin{definition}
    For a space \(X\) and \(x \in X\), let \(\Omega_{X,x} = \{ A \subseteq X : x \in \mathrm{cl}(A) \setminus A \}\)
    where \(\mathrm{cl}(A)\) is the closure of \(A\) in \(X\).
    Then the space \(X\) has \emph{countable fan-tightness at \(x\)} provided that
    \(X \models \mathsf S_{\mathrm{fin}}(\Omega_{X,x},\Omega_{X,x})\).
    We say that \(X\) has \emph{countable fan-tightness} if \(X\) has countable fan-tightness
    at \(x\) for every \(x \in X\).
\end{definition}
We will also address the single-selection analog.
\begin{definition}
    For a space \(X\) and \(x \in X\), \(X\) has \emph{countable strong fan-tightness at \(x\)}
    if \(X \models \mathsf S_1(\Omega_{X,x},\Omega_{X,x})\).
    If \(X\) has countable strong fan-tightness at \(x\) for every \(x \in X\),
    then \(X\) is said to have \emph{countable strong fan-tightness}.
\end{definition}
Ko{\v{c}}inac and Scheepers, in \cite{COOC7}, extend the characterization of \cite{ArkhangelskiiHurewicz}
to include the so-called Reznichenko property, which, according to \cite{WeaklyFrechet},
was introduced by Reznichenko at a seminar
in 1996.
\begin{definition}
    For a space \(X\) and \(x \in X\), \(X\) has the \emph{Reznichenko property at \(x\)} if
    \(\Omega_{X,x} \cap [X]^\omega = \Omega_{X,x}^{\mathrm{gp}}\).
    If \(X\) has the Reznichenko property at \(x\) for every \(x \in X\), we say that \(X\) is \emph{Reznichenko}.
\end{definition}
As \(C_{\mathcal A}(X)\) is homogeneous, we can address the properties listed above
at any particular point and establish them for the entire space;
for convenience, we focus on \(\mathbf 0\), the function which is constantly \(0\).

Combining Corollary \ref{cor:MainTheoremFinite} and \cite[Thm. 21]{COOC7}, we obtain:
\begin{theorem}
    For any Tychonoff space \(X\), the following are equivalent:
    \begin{itemize}
        \item
        \(\mathcal P_{\mathrm{fin}}(X)\) is Hurewicz.
        \item
        \(\mathcal P_{\mathrm{fin}}(X) \models \mathsf{S}_{\mathrm{fin}}(\Omega,\Omega^{\mathrm{gp}})\).
        \item
        \(X \models \mathsf{S}_{\mathrm{fin}}(\Omega,\Omega^{\mathrm{gp}})\).
        \item
        Every finite power of \(X\) is Hurewicz.
        \item
        One does not have a winning strategy in \(\mathsf G_{\mathrm{fin}}(\Omega_X,\Omega_X^{\mathrm{gp}})\).
        \item
        \(C_p(X) \models \mathsf S_{\mathrm{fin}}(\Omega_{C_p(X),\mathbf 0},\Omega_{C_p(X),\mathbf 0}^{\mathrm{gp}})\).
        \item
        \(C_p(X)\) is Reznichenko and has countable fan-tightness.
    \end{itemize}
\end{theorem}
Similarly, combining Corollary \ref{cor:MainFiniteSingle}, \cite{KocinacScheepersReznichenko}, and \cite[Thm. 26]{COOC7}, we obtain:
\begin{theorem}
    For any Tychonoff space \(X\), the following are equivalent:
    \begin{itemize}
        \item
        \(\mathcal P_{\mathrm{fin}}(X)\) is Gerlits-Nagy.
        \item
        \(\mathcal P_{\mathrm{fin}}(X) \models \mathsf{S}_1(\Omega,\Omega^{\mathrm{gp}})\).
        \item
        \(X \models \mathsf{S}_1(\Omega,\Omega^{\mathrm{gp}})\).
        \item
        Every finite power of \(X\) is Gerlits-Nagy.
        \item
        One does not have a winning strategy in \(\mathsf G_1(\Omega_X,\Omega_X^{\mathrm{gp}})\).
        \item
        \(C_p(X) \models \mathsf S_1(\Omega_{C_p(X),\mathbf 0},\Omega_{C_p(X),\mathbf 0}^{\mathrm{gp}})\).
        \item
        \(C_p(X)\) is Reznichenko and has countable strong fan-tightness.
    \end{itemize}
\end{theorem}
We now turn our attention to \(C_k(X)\).
For a set \(X\), we will let \(\wp(X)\) denote the power set of \(X\).
\begin{definition}
    Let \(X\) be a Tychonoff space.
    We define \(\mathfrak u : C(X) \times \omega \to \mathscr T_X \cup \{ X \}\)
    by the rule \(\mathfrak u(f,n) = f^{-1}((-2^{-n},2^{-n}))\).
    We then define \(\mathbb U : \Omega_{C_k(X),\mathbf 0} \times \omega \to \wp(\mathscr T_X)\) by the rule
    \[\mathbb U(\mathscr F, n) = \{ U \in \mathscr T_X : (\exists K \in K(X))(\exists f \in \mathscr F)\ K \subseteq U \subseteq \mathfrak u(f,n) \}.\]
    Also define \(\mathscr D : \mathcal K_X \to \wp(C(X))\),
    by the rule
    \[\mathscr D(\mathscr U) = \{ f \in C(X) : (\exists U \in \mathscr U)\ f[X \setminus U] = \{1\} \}.\]
\end{definition}
\begin{lemma} \label{lem:CoversBlades}
    Let \(X\) be a Tychonoff space.
    \begin{enumerate}[label=(\roman*)]
        \item \label{BladesToCovers}
        If \(\mathscr F \in \Omega_{C_k(X),\mathbf 0}\), then, for any \(n \in \omega\),
        \(\mathbb U(\mathscr F,n) \in \mathcal K_X\).
        \item \label{CoversToBlades}
        If \(\mathscr U \in \mathcal K_X\), then \(\mathscr D(\mathscr U) \in \Omega_{C_k(X),\mathbf 0}\).
    \end{enumerate}
\end{lemma}
\begin{proof}
    \ref{BladesToCovers}:
    Let \(K \in K(X)\) and consider the neighborhood \([\mathbf 0; K, 2^{-n}]\) of \(\mathbf 0\).
    We can choose \[f \in \mathscr F \cap [\mathbf 0; K, 2^{-n}].\]
    For each \(x \in K\), \(|f(x)| < 2^{-n}\) so we see that \[\sup\{ |f(x)| : x \in K \} < 2^{-n}.\]
    So \(K \subseteq f^{-1}((-2^{-n},2^{-n})) = \mathfrak u(f,n)\).
    If \(\mathfrak u(f,n) = X\), we can let \(U \in \mathscr T_X\) be so that \(K \subseteq U \subseteq \mathfrak u(f,n)\).
    Then \(U \in \mathbb U(\mathscr F,n)\).
    Otherwise, \(\mathfrak u(f,n) \in \mathbb U(\mathscr F,n)\).

    \ref{CoversToBlades}:
    Consider any neighborhood \([\mathbf 0; K, \varepsilon]\) of \(\mathbf 0\) in \(C_k(X)\).
    Since \(\mathscr U \in \mathcal K_X\), we can find \(U \in \mathscr U\) so that \(K \subseteq U\).
    As \(X\) is Tychonoff, we can find \(f \in C(X)\) so that \(f[K] = \{0\}\) and \(f[X \setminus U] = \{1\}\).
    Now, \[f \in \mathscr D(\mathscr U) \cap [\mathbf 0;K,\varepsilon],\]
    which completes the proof.
\end{proof}
We remark that, as reported in \cite{KocinacClosure2003}, it was shown in \cite{LinLiuTeng1994}
that \(X \models \mathsf S_{\mathrm{fin}}(\mathcal K, \mathcal K)\) if and only if \(C_k(X)\) has countable fan-tightness.
The single-selection analog of this equivalence was established in \cite{KocinacClosure2003}.
We now establish the groupable analogues of these equivalences to include the Reznichenko property on \(C_k(X)\) which answers
Problem 4.3 of \cite{KocinacClosure2003} in the affirmative.
We provide the complete proof (in the finite-selection context) for the convenience of the reader.
\begin{theorem} \label{thm:ReznichenkoStuff}
    For any Tychonoff space, the following are equivalent:
    \begin{enumerate}[label=(\roman*)]
        \item \label{CkGroundGroup}
        \(X \models \mathsf S_{\mathrm{fin}}(\mathcal K, \mathcal K^{\mathrm{gp}})\)
        \item  \label{KGameOne}
        One does not have a winning strategy in \(\mathsf G_{\mathrm{fin}}(\mathcal K_X, \mathcal K_X^{\mathrm{gp}})\).
        \item \label{ReznichenkoFanTightness}
        \(C_k(X)\) is Reznichenko and has countable fan-tightness.
        \item \label{CkGroup}
        \(C_k(X) \models \mathsf S_{\mathrm{fin}}(\Omega_{C_k(X),\mathbf 0},\Omega^{\mathrm{gp}}_{C_k(X),\mathbf 0})\).
    \end{enumerate}
\end{theorem}
\begin{proof}
    The equivalence of \ref{CkGroundGroup} and \ref{KGameOne} is established in Corollary \ref{cor:MainTheorem}.

    \ref{CkGroundGroup}\(\implies\)\ref{ReznichenkoFanTightness}:
    We first show that \(C_k(X)\) has countable fan-tightness.
    So let \(\langle \mathscr F_n : n \in \omega \rangle\) be a sequence of \(\Omega_{C_k(X),\mathbf 0}\).
    We then let \(\mathscr U_n = \mathbb U(\mathscr F_n, n)\).
    Note that \(\mathscr U_n \in \mathcal K_X\) by Lemma \ref{lem:CoversBlades}.
    As \(X \models \mathsf S_{\mathrm{fin}}(\mathcal K, \mathcal K^{\mathrm{gp}})\),
    we can choose \(\mathscr G_n \in [\mathscr U_n]^{<\omega}\) so that
    \[\bigcup_{n\in\omega} \mathscr G_n \in \mathcal K_X^{\mathrm{gp}}.\]
    Let \(\phi : \bigcup_{n\in\omega} \mathscr G_n \to \omega\) be as guaranteed by the definition of
    groupability.

    For each \(n \in \omega\), define \(\mathfrak f_n : \mathscr G_n \to \mathscr F_n\) to be so that,
    for \(U \in \mathscr G_n\), \(U \subseteq \mathfrak u(\mathfrak f_n(U), n)\).
    Let \(\mathscr H_n = \mathfrak f_n[\mathscr G_n]\) and observe that \(\mathscr H_n \in [\mathscr F_n]^{<\omega}\).

    To see that \(\bigcup_{n\in\omega} \mathscr H_n \in \Omega_{C_k(X),\mathbf 0}\), consider any basic neighborhood
    \([\mathbf 0; K , \varepsilon]\) of \(\mathbf 0\).
    Let \(m \in \omega\) be so that \(2^{-m} < \varepsilon\) and consider
    \[M := \max \left\{ \phi(U) : U \in \bigcup_{k<m} \mathscr G_k \right\}.\]
    By the groupability criterion, \(\bigcup\{ \phi^{-1}(n) : n > M \} \in \mathcal K_X.\)
    So we can find \(\ell_0 > M\) and \(U \in \phi^{-1}(\ell_0)\) so that \(K \subseteq U\).
    Since \(\ell_0 > M\), we know that \(U \not\in \bigcup_{k < m} \mathscr G_k\).
    So we can find \(\ell \geq m\) so that \(U \in \mathscr G_\ell\).
    It follows that
    \[K \subseteq U \subseteq \mathfrak u(\mathfrak f_\ell(U) , \ell) = \mathfrak f_\ell(U)^{-1}((-2^{-\ell},2^{-\ell})).\]
    As \(\ell \geq m\) and \(2^{-m} < \varepsilon\), we see that
    \[\mathfrak f_\ell(U) \in [\mathbf 0; K , \varepsilon] \cap \bigcup_{n\in\omega} \mathscr H_n.\]

    Now, we show that \(C_k(X)\) is Reznichenko.
    So let \(\mathscr F = \{ f_n : n \in \omega \} \in \Omega_{C_k(X),\mathbf 0}\).
    By Corollary \ref{cor:MainTheorem}, we know that \(\mathbb K(X)\) is Hurewicz and that,
    by Theorem \ref{thm:HurewiczGameChar}, One has no winning strategy in the Hurewicz game
    on \(\mathbb K(X)\).
    So we define a strategy \(\sigma\) for One in the Hurewicz game on \(\mathbb K(X)\).

    Set \(\mathscr F_0 = \mathscr F\).
    For \(K \in K(X)\), let \(\mathfrak f_0(K) \in [\mathbf 0; K, 1] \cap \mathscr F_0\).
    Then let \(U_0(K) \in \mathscr T_X\) be so that \[K \subseteq U_0(K) \subseteq \mathfrak f_0(K)^{-1}((-1,1)).\]
    Let \(\sigma_0(\langle \rangle) = \{ U_0(K) : K \in K(X) \}\) and
    \(\sigma(\langle\rangle) = \{ [U] : U \in \sigma_0(\langle \rangle) \}\), and notice that \(\sigma(\langle\rangle)\)
    is an open cover of \(\mathbb K(X)\).
    Also fix a choice function \(\gamma_0 : \sigma_0(\langle \rangle) \to \omega\) so that
    \(U \subseteq f_{\gamma_0(U)}^{-1}((-1,1))\) and observe that
    \(U \mapsto [U]\), \(\sigma_0(\langle \rangle) \to \sigma(\langle \rangle)\), is a bijection.

    For \(k \in\omega\), suppose we have \(\langle \mathscr V^\ast_j : j < k \rangle\), \(\langle \mathscr V_j : j < k \rangle\),
    \(\langle M_j : j < k \rangle\), \(\sigma_0(\langle \mathscr V_j : j < k \rangle)\),
    \(\sigma(\langle \mathscr V_j : j < k \rangle)\), \(\langle \gamma_j : j \leq k \rangle\),
    and \(\langle \mathscr F_j : j \leq k \rangle\) defined.
    For any \(\mathscr V_k \in [\sigma(\langle \mathscr V_j : j < k \rangle)]^{<\omega}\),
    let \(\mathscr V_k^\ast \in [\sigma_0(\langle \mathscr V_j : j < k \rangle)]^{<\omega}\) be so that
    \(\{ [U] : U \in \mathscr V_k^\ast \} = \mathscr V_k\).
    Also, let \(M_k = \max \gamma_k[\mathscr V^\ast_k]\) and \(\mathscr F_{k+1} = \{ f_n : n > M_k \}\).

    For \(K \in K(X)\), let \(\mathfrak f_{k+1}(K) \in [\mathbf 0;K,2^{-k-1}] \cap \mathscr F_{k+1}\).
    Then let \(U_{k+1}(K) \in \mathscr T_X\) be so that
    \[K \subseteq U_{k+1}(K) \subseteq \mathfrak f_{k+1}(K)^{-1}((-2^{-k-1},2^{-k-1})).\]
    Define
    \[\sigma_0(\langle \mathscr V_j : j \leq k \rangle) = \{ U_{k+1}(K) : K \in K(X) \}\]
    and \(\sigma(\langle \mathscr V_j : j \leq k \rangle ) = \{ [U] : U \in \sigma_0(\langle \mathscr V_j : j \leq k \rangle) \}\).
    Note that \(U \mapsto [U]\), \(\sigma_0(\langle \mathscr V_j : j \leq k \rangle) \to \sigma(\langle \mathscr V_j : j \leq k \rangle )\),
    is a bijection.
    Observe also that \(\sigma(\langle \mathscr V_j : j \leq k \rangle )\) is an open cover of \(\mathbb K(X)\)
    and fix a choice function \(\gamma_{k+1} : \sigma_0(\langle \mathscr V_j : j \leq k \rangle) \to \{n \in \omega : n > M_k\}\)
    so that \[U \subseteq f_{\gamma_{k+1}(U)}^{-1}((-2^{-k-1},2^{-k-1})).\]

    This defines a strategy \(\sigma\) for One in the Hurewicz game on \(\mathbb K(X)\).
    Since \(\sigma\) is not winning, there is a sequence \(\langle \mathscr V_n : n \in \omega \rangle\)
    of choices with a corresponding increasing sequence \(\langle M_n : n \in \omega \rangle\) of naturals as described above
    so that
    \[
        (\forall K \in \mathbb K(X))(\exists m \in \omega)(\forall n \geq m)\ K \in \bigcup \mathscr V_n.
    \]

    Define \(\phi : \mathscr F \to \omega\) by \(\phi(f_n) = \min \{ k \in \omega : n \leq M_k \}\).
    We show that this meets the criterion for groupability.
    Let \(J \subseteq \omega\) be infinite and consider an arbitrary neighborhood \([\mathbf 0; K , \varepsilon]\)
    of \(\mathbf 0\).
    We can find \(\ell \in J\) so that \(K \in \bigcup \mathscr V_\ell\),
    and \(2^{-\ell} < \varepsilon\).
    So there is some \(U \in \mathscr V_\ell^\ast\) with \(K \in [U]\), which is to say that \(K \subseteq U\).
    By construction, \(M_{\ell-1} < \gamma_\ell(U) \leq M_\ell\), so \(\phi(f_{\gamma_\ell(U)}) = \ell\).
    Observe further that
    \[K \subseteq U \subseteq f_{\gamma_\ell(U)}^{-1}((-2^{-\ell},2^{-\ell})).\]
    That is, that \(f_{\gamma_\ell(U)} \in [\mathbf 0;K,\varepsilon]\).
    Therefore, \(\bigcup \{ \phi^{-1}(n) : n \in J \} \in \Omega_{C_k(X),\mathbf 0}\).

    The implication \ref{ReznichenkoFanTightness}\(\implies\)\ref{CkGroup} is evident.

    \ref{CkGroup}\(\implies\)\ref{CkGroundGroup}:
    We will show that \(X \models \mathsf S_{\mathrm{fin}}(\mathcal K, \mathcal K)\) and that
    \(\mathcal K_X^{\mathrm{gp}} = \mathcal K_X \cap [\mathscr T_X]^\omega\), which, by Corollary \ref{cor:MainTheorem},
    establishes the implication.

    To show that \(X \models \mathsf S_{\mathrm{fin}}(\mathcal K, \mathcal K)\),
    let \(\langle \mathscr U_n : n \in \omega \rangle\) be a sequence of \(k\)-covers of \(X\).
    By Lemma \ref{lem:CoversBlades}, \(\langle \mathscr D(\mathscr U_n) : n \in \omega \rangle\)
    is a sequence of \(\Omega_{C_k(X),\mathbf 0}\).
    We can then choose, for each \(n \in \omega\), \(\mathscr G_n \in [\mathscr D(\mathscr F_n)]^{<\omega}\)
    so that \(\bigcup_{n\in\omega} \mathscr G_n \in \Omega_{C_p(X),\mathbf 0}\).
    Now, for each \(\mathscr G_n\), let \(\mathscr H_n \in [\mathscr U_n]^{<\omega}\)
    be so that, for \(f \in \mathscr G_n\), there is some \(U \in \mathscr H_n\) so that
    \(f[X \setminus U] = \{1\}\).
    We show that \(\bigcup_{n\in\omega} \mathscr H_n \in \mathcal K_X\).
    So let \(K \in K(X)\) and consider the neighborhood \([\mathbf 0; K, 1/2]\) of \(\mathbf 0\).
    We can find \(n \in \omega\) and \(f \in \mathscr G_n\) so that \(f \in [\mathbf 0; K , 1/2]\).
    Then, for \(U \in \mathscr H_n\) with \(f[X \setminus U] = \{1\}\), we see that \(K \subseteq U\).
    Hence, \(\bigcup_{n\in\omega} \mathscr H_n \in \mathcal K_X\).

    To show that every countable \(k\)-cover of \(X\) is groupable, let
    \(\mathscr U = \{ U_n : n \in \omega \} \in \mathcal K_X\).
    We will define a strategy for One in \(\mathsf G_{\mathrm{fin}}(\mathcal K_X, \mathcal K_X)\).

    For each \(K \in K(X)\) and \(U \in \mathscr U\) so that \(K \subseteq U\),
    let \(f_{K,U} \in C(X)\) be so that \(f_{K,U}[K] = \{0\}\) and \(f_{K,U}[X \setminus U] = \{1\}\).
    It follows that
    \[\mathscr F_0 := \{ f_{K,U} : K \in K(X), U \in \mathscr U, K \subseteq U \} \in \Omega_{C_k(X),\mathbf 0}.\]
     By \(C_k(X) \models \mathsf S_{\mathrm{fin}}(\Omega_{C_k(X),\mathbf 0}\Omega^{\mathrm{gp}}_{C_k(X),\mathbf 0})\),
    we can find \(\{ g_{n} : n \in \omega \} \subseteq \mathscr F_0\) and an increasing sequence
    \(\langle M_{n} : n \in \omega\rangle\) of naturals so that
    \(\psi : \{ g_{n} : n \in \omega \} \to \omega\) defined by
    \[\psi(g_{n}) = \min \{ \ell \in \omega : n \leq M_{\ell} \}\] satisfies the condition of groupability.
    In particular, note that, for any \(K \in K(X)\), there exists some \(m \in \omega\) so that,
    for all \(n \geq m+1\), \[[\mathbf 0; K , 1/2] \cap \{ g_{\ell} : M_{n-1} < \ell \leq M_{n} \} \neq \emptyset.\]

    Define \(\gamma : \omega \to \omega\) to be so that
    \[\gamma(n) = \min\{ k \in \omega : g_n[X \setminus U_{k}] = \{1\}\}\]
    and let
    \[\sigma(\langle \rangle) = \mathscr W := \{ U_{\gamma(n)} : n \in \omega \}.\]
    Observe that \(\sigma(\langle \rangle) \in \mathcal K(X)\).

    For \(k \in \omega\), suppose we have \(\langle \mathscr V_j : j < k \rangle\),
    \(\langle \nu_j : j < k \rangle\), \(\langle \mu_j : j < k \rangle\), and
    \(\sigma(\langle \mathscr V_j : j < k \rangle)\) defined.
    For \(\mathscr V_k \in [\sigma(\langle \mathscr V_j : j < k \rangle)]^{<\omega}\),
    consider \(\Lambda \in [\omega]^{<\omega}\) so that
    \(\mathscr V_k = \{ U_{\gamma(n)} : n \in \Lambda \}\).
    Let \(\mu_k = \max \{ \psi(g_n) : n \in \Lambda \}\) and
    \(\nu_k = \max \{ \gamma(n) : n \leq M_{\mu_k} \}\).
    Now, define
    \[\sigma(\langle \mathscr V_j : j \leq k \rangle) = \{ U_{n} : n > \nu_k \} \cap \mathscr W \in \mathcal K_X.\]

    This completes the definition of \(\sigma\).
    By Theorem \ref{kPawlikowski}, \(\sigma\) cannot be a winning strategy.
    So there is a play \(\langle \mathscr V_n : n \in \omega \rangle\) by Two that beats \(\sigma\).
    We also have the auxiliary sequences \(\langle \nu_n : n \in \omega \rangle\)
    and \(\langle \mu_n : n \in \omega \rangle\) as defined above.

    We show that the sequences \(\langle \nu_n : n \in \omega \rangle\)
    and \(\langle \mu_n : n \in \omega \rangle\) are strictly increasing.
    Let \(k \in \omega\) and note that
    \[\sigma(\langle \mathscr V_j : j \leq k \rangle) = \{ U_n : n > \nu_k \} \cap \mathscr W.\]
    Now, consider \(\Lambda_0 , \Lambda_1 \in [\omega]^{<\omega}\) so that
    \(\mathscr V_k = \{ U_{\gamma(n)} : n \in \Lambda_0\}\) and
    \(\mathscr V_{k+1} = \{ U_{\gamma(n)} : n \in \Lambda_1 \}\).
    For \(\ell \in \Lambda_1\), notice that \(\gamma(\ell) > \nu_k\).
    Since \(\nu_k = \max \{ \gamma(n) : n \leq M_{\mu_k}\}\), we see that
    \(\ell > M_{\mu_k}\).
    Hence, \[\psi(g_\ell) > \max\{\psi(g_n) : n \in \Lambda_0\} = \mu_k\] and so \(\mu_{k+1} > \mu_k\).
    Moreover, the inequalities \(\nu_k < \gamma(\ell)\) and
    \(\ell \leq M_{\psi(g_\ell)} \leq M_{\mu_{k+1}}\)
    imply that \(\nu_k < \nu_{k+1}\).

    So we define \(\phi : \mathscr U \to \omega\) by the rule
    \(\phi(U_n) = \min\{ k \in \omega : n \leq \nu_k \}.\)
    We show that \(\phi\) satisfies the criterion of groupability.
    So let \(J \subseteq \omega\) be infinite, let \(K \in K(X)\) be arbitrary, and consider
    \([\mathbf 0 ; K , 1/2 ]\).
    By our definition of \(\{g_n : n \in \omega\}\),
    we can let \(j \in J\), \(j > 0\), be so that
    \[[\mathbf 0;K,1/2] \cap \{ g_\ell : M_{\mu_{j-1}} < \ell \leq M_{\mu_j} \} \neq\emptyset.\]
    So let \(\ell \in (M_{\mu_{j-1}} , M_{\mu_j}]\) be so that \(g_\ell \in [\mathbf 0; K, 1/2]\)
    and observe that \(K \subseteq U_{\gamma(\ell)}\).
    Since \(\ell \leq M_{\mu_j}\), we see that \(\gamma(\ell) \leq \nu_k\).
    Moreover, since the interval \((M_{\mu_{j-1}},M_{\mu_j}]\) was defined in the \(j^{\mathrm{th}}\) inning
    according to \(\sigma\), we know that \(\gamma(\ell) > \nu_{j-1}\).
    Therefore, \(\phi(U_{\gamma(\ell)}) = j\) and the proof is complete.
\end{proof}

We end with the single-selection analog to Theorem \ref{thm:ReznichenkoStuff}.
The proof is omitted as it is nearly identical to the proof of Theorem \ref{thm:ReznichenkoStuff}.
The key thing to notice is that the transferal of finite-selections above allow one
to transfer single-selections in the same way.
Also, the use of Corollary \ref{cor:MainTheorem} in the proof needs to be replaced with Corollary \ref{cor:GroupPawlikowski}.
\begin{theorem}
    For any Tychonoff space, the following are equivalent:
    \begin{enumerate}[label=(\roman*)]
        \item
        \(X \models \mathsf S_1(\mathcal K, \mathcal K^{\mathrm{gp}})\)
        \item
        One does not have a winning strategy in \(\mathsf G_{1}(\mathcal K_X, \mathcal K_X^{\mathrm{gp}})\).
        \item
        \(C_k(X)\) is Reznichenko and has countable strong fan-tightness.
        \item
        \(C_k(X) \models \mathsf S_1(\Omega_{C_k(X),\mathbf 0},\Omega^{\mathrm{gp}}_{C_k(X),\mathbf 0})\).
    \end{enumerate}
\end{theorem}

\providecommand{\bysame}{\leavevmode\hbox to3em{\hrulefill}\thinspace}
\providecommand{\MR}{\relax\ifhmode\unskip\space\fi MR }
\providecommand{\MRhref}[2]{%
  \href{http://www.ams.org/mathscinet-getitem?mr=#1}{#2}
}
\providecommand{\href}[2]{#2}

\end{document}